\documentclass{article}

\usepackage{amsmath,amssymb,amsthm}
\usepackage{tikz,tikz-cd}
\usetikzlibrary{calc,intersections}

\theoremstyle{plain}
\newtheorem{thm}{Theorem}[section]
\newtheorem{prop}[thm]{Proposition}
\newtheorem{lem}[thm]{Lemma}
\newtheorem{cor}[thm]{Corollary}

\theoremstyle{definition} 
\newtheorem{dfn}[thm]{Definition}

\theoremstyle{remark}
\newtheorem{rem}[thm]{Remark}
\newtheorem{note}[thm]{Notation}
\newtheorem{q}[thm]{Question}

\renewenvironment{itemize}
{%
\begin{list}{\parbox{1em}{$\bullet$}}
{%
\setlength{\topsep}{0em}
\setlength{\itemindent}{0em}
\setlength{\leftmargin}{2.5em}
\setlength{\rightmargin}{0em}
\setlength{\labelsep}{.5em}
\setlength{\labelwidth}{2em}
\setlength{\itemsep}{0em}
\setlength{\parsep}{0em}
\setlength{\listparindent}{0em}
}
}{\end{list}}

\begin{document}
\title{On coarse geometric aspects of Hilbert geometry}
\author{Ryosuke Mineyama\hspace{-2pt}
\thanks{Department of Mathematics, Graduate School of Science, Osaka University,
		Toyonaka, Osaka 560-0043, Japan, 
		\texttt{r-mineyama@cr.math.sci.osaka-u.ac.jp}}
\thanks{Partly supported by Grant-in-Aid for JSPS Fellows No.13J01771.},
		\quad
		Shin-ichi Oguni\hspace{-2pt}
\thanks{Department of Mathematics, Faculty of Science, Ehime University,
		Matsuyama 790-8577, Japan, 
		\texttt{oguni@math.sci.ehime-u.ac.jp}.}
\thanks{Partly supported by JSPS Grant-in-Aid for Young Scientists (B) No.24740045, 16K17595.}}
\date{\empty}

\maketitle

\begin{abstract}
    We begin a coarse geometric study of Hilbert geometry.
	Actually we give a necessary and sufficient condition for the natural boundary of a Hilbert geometry 
	to be a corona, which is a nice boundary in coarse geometry. 
	In addition, we show that any Hilbert geometry is uniformly contractible and with coarse bounded geometry. 
	As a consequence of these we see that the coarse Novikov conjecture 
	holds for a Hilbert geometry with a mild condition. Also we show that 
	the asymptotic dimension of any two-dimensional Hilbert geometry is just two.
	This implies that the coarse Baum-Connes conjecture holds for any two-dimensional Hilbert geometry
	via Yu's theorem.
	
	\vspace{8pt}
	
	\vspace{4pt}
	\noindent 2010MSC: 51F99. 
\end{abstract}

%------------------------------------------
%	1:Introduction
%------------------------------------------

\section{Introduction}
Coarse geometry studies metric spaces
by neglecting small-scale structures and thus focusing on large-scale ones \cite{Nowak-Yu,Roe2}.
For example one does not distinguish all bounded metric spaces, and also
identifies the Euclidean space $\mathbb{R}^n$ with its integer lattice $\mathbb{Z}^n$.
The coarse Novikov conjecture and its stronger version, that is,
the coarse Baum-Connes conjecture \cite{H-R95, Roe1, Yu95} are
interesting by the following two reasons at least. 

\begin{enumerate}
	\item These have applications to differential topology of manifolds related to the Novikov conjecture \cite{H-R00}.
	
	\item These are targets of applications of some important coarse
	geometric properties, for instance having finite asymptotic dimension and 
	coarse embeddability into a Hilbert space, which are sufficient conditions 
	for a metric spaces with coarse bounded geometry to satisfy 
	the coarse Baum-Connes conjecture, respectively, showed by Yu \cite{Yu98, Yu00}.
\end{enumerate}

\noindent Note that whether the coarse Novikov conjecture (resp. the coarse Baum-Connes conjecture) 
holds or not is invariant under the equivalence in the sense of coarse geometry. 
There are a number of positive results for these conjectures based on the above sufficient conditions.  
Also by other approaches, it is known that 
``non-positively curved'' spaces like proper geodesic Gromov-hyperbolic spaces 
and CAT(0) spaces \cite{Bridson-Haefliger} satisfy the coarse Baum-Connes conjecture \cite{F-O16, H-R95, Willett}.
Especially boundaries called coronae played an essential role for the case of  
proper geodesic Gromov-hyperbolic spaces \cite{H-R95}. 

A Hilbert geometry is defined on a bounded convex domain $X$
in the Euclidean space $\mathbb{R}^n$ endowed with the Hilbert metric $d$. 
This is a classical and naive geometric object, 
the study of which is still being developed actively \cite{P-T}. 
Now we can formulate somewhat new questions from a coarse geometric point of view. 
The following are typical ones. 

\begin{q}\label{question1}
\begin{itemize}
	\item[(1)]
	Does a Hilbert geometry satisfy the coarse Novikov conjecture or the coarse Baum-Connes conjecture?
	\item[(2)]
	Can we determine the asymptotic dimension of a Hilbert geometry? In particular is it finite?
    \item[(3)]
    Is a Hilbert geometry coarsely embeddable into a Hilbert space?
\end{itemize}
\end{q}

We expect that the coarse geometric investigation expands 
the study of Hilbert geometry and that, conversely, 
Hilbert geometry provides interesting examples for coarse geometry.  

An $n$-dimensional unit open ball equipped with the Hilbert metric is 
the projective model of the $n$-dimensional hyperbolic space,  
any $n$-dimensional polygon with the Hilbert metric is bi-Lipschitz equivalent to 
the Euclidean space $\mathbb{R}^n$ \cite{Veronicos}.
Thus Question \ref{question1} is positively solved in such spaces \cite{Nowak-Yu,Roe2}, 
but it is not known for general Hilbert geometries.  
A Hilbert geometry $(X,d)$ is not necessarily either CAT(0) or Gromov-hyperbolic. 
Indeed an investigation of Kelly and Straus \cite{Kelly-Straus} implies that 
$(X,d)$ is CAT(0) if and only if $X$ is an ellipsoid.
Besides if $(X,d)$ is Gromov-hyperbolic then 
the boundary $\partial X=\overline{X}\setminus X$ must be of differentiability class $C^1$
where $\overline{X}$ is the closure of $X$ in $\mathbb{R}^n$ \cite{Karlsson-Noskov}.  
Hence it is difficult to answer Question \ref{question1} (1) by only using known results in coarse geometry. 
On the other hand the above facts indicate that 
the shape or smoothness of the boundary $\partial X$ dominates geometric properties of the Hilbert geometry $(X,d)$.  
Now a question naturally occurs. 

\begin{q}\label{question2}
	Is the natural boundary of a Hilbert geometry a corona? 
\end{q}

In this paper, we discuss Questions \ref{question1} and \ref{question2}. 
The following first main theorem completely answers Question \ref{question2}. 

\begin{thm}\label{thm:main-1}
	Let $X \subset \mathbb{R}^n$ be a non-empty bounded convex domain. 
	Then the boundary $\partial X$ is a corona of the Hilbert geometry $(X,d)$ 
	if and only if the closure $\overline{X}$ is properly convex in $\mathbb{R}^n$.
\end{thm}

\noindent
This theorem enables us to answer Question \ref{question1} (1) positively in a certain case by 
combining with Propositions \ref{prop:3} and \ref{prop:1},  
which claim that any Hilbert geometry is uniformly contractible and with coarse bounded geometry.
 
\begin{cor}\label{cor:1}
	Let $X \subset \mathbb{R}^n$ be a non-empty bounded convex domain and 
	the closure $\overline{X}$ be properly convex in $\mathbb{R}^n$. 
	Then the Hilbert geometry $(X,d)$ satisfies the coarse Novikov conjecture. 
\end{cor}

\noindent
Our second main theorem answers partially Question \ref{question1} (2).
 
\begin{thm}\label{thm:main-2}
	The asymptotic dimension of any $2$-dimensional Hilbert geometry is equal to $2$.   
\end{thm}

\noindent
Putting together with Proposition \ref{prop:1}, we see that
any $2$-dimensional Hilbert geometry satisfies the coarse Baum-Connes conjecture via Yu's result \cite{Yu98}. 
Another consequence of this theorem is that 
any $2$-dimensional Hilbert geometry can be coarsely embedded into a Hilbert space.
This follows from a well-known fact that
any metric space with coarse bounded geometry can be coarsely embedded into a Hilbert space
whenever the space has finite asymptotic dimension \cite{Roe2}. 

\begin{note}
	We collect some notations which will frequently appear in this paper.
	We always assume that $2\le n < \infty$.
	\begin{itemize}
	\item
		A \emph{line} through $a,b \in \mathbb{R}^n$ is a set 
		$\{\ t a + (1-t) b \in \mathbb{R}^n \ \vert\ t \in \mathbb{R}\ \}$.
	\item
		For $x,y \in \mathbb{R}^n$, $[x,y]$ denotes the (directed) segment from $x$ to $y$. 
	\item
		For $x,y \in \mathbb{R}^n$, $|x y|$ denotes the Euclidean distance between $x$ and $y$.
	\item
		We write the closure of an open set $A \subset \mathbb{R}^n$ as $\overline{A}$
		and put $\partial A := \overline{A} \setminus A$.
		For a closed set $B \subset \mathbb{R}^n$, we put $\partial B := B \setminus \mathrm{int}(B)$ 
		where $\mathrm{int}(B)$ is the interior of $B$.
	\item
		For a bounded convex domain $X \subset \mathbb{R}^n$, 
		a {\em{chord}} $[x',y']$ is a (directed) segment connecting 
		two boundary points $x',y' \in \partial X$.
	\item
		For a Hilbert geometry $(X,d)$ (defined in \S \ref{sec:hilb}), 
		$B(x,r)$ denotes the closed ball of radius $r >0$ centered at $x \in X$ with respect to $d$.
		If we do not wish to specify the center, we simply denote by $B_r$ a closed ball of radius $r$.
		When we consider a Euclidean closed $r$-ball
		 (resp. a closed $r$-ball in a general metric space $(Y,d_Y)$),
		we write it as $B_{euc}(x,r)$ (resp. $B_Y(x,r)$).
	\end{itemize} 
\end{note}

%------------------------------------------
%	2:Preliminaries
%------------------------------------------
\section{Preliminaries}
In this section we recall fundamental facts, especially about Hilbert geometry.
Topics related to coarse geometry are dealt on each occasion
(Definitions \ref{uniformly contractible}, \ref{coarse bounded geometry}, 
\ref{corona}, \ref{asymptotic dimension}).
We refer to a book by Roe \cite{Roe2} or 
the book by Nowak and Yu \cite{Nowak-Yu} for a comprehensive account of these terminologies.

%	2-1
%------------------------------------------
\subsection{Convexity in the Euclidean space}\label{sec:pre}
A metric space $(Y,d_Y)$ is said to be a \emph{geodesic space} (resp. \emph{uniquely geodesic space})
if any two points are joined by a geodesic (resp. a unique geodesic).
Here, a \emph{geodesic} is (the image of) an isometric embedding of a closed interval of $\mathbb{R}$ into $Y$.

For a uniquely geodesic space $(Y,d_Y)$,
we say that a subset $A$ of $Y$ is \emph{convex} (resp. \emph{properly convex}) if for every $x$ and $y$ in $A$,
any point $z$ distinct from $x,y$ on the geodesic joining $x,y$ is contained in $A$ (resp. the interior of $A$).
See \cite[Definitions 2.5.2, 2.5.6]{Papadopoulos}.
If $(Y,d_Y)$ is \emph{proper}, that is, any bounded closed set is compact, 
then the closure of a convex set is also convex \cite[Proposition 2.5.3]{Papadopoulos}. 

\begin{lem}\label{lem:4}
	For a bounded domain $X$ in $\mathbb{R}^n$, 
	the closure $\overline{X}$ is properly convex if and only if 
	$X$ is convex and its boundary $\partial X$ does not include any non-trivial segment.
\end{lem}

\begin{proof}
	The necessary condition is obvious.
	In order to prove the converse, we assume that $[x,y] \cap \partial X$ contains $z \neq x,y$.
	Since $\partial X$ does not include any non-trivial segment, there exist two points
	$x' \in [x,z]$ and $y' \in [z,y]$ with $x',y' \in X$.
	Then $z \in [x',y'] \not\subset X$. 
	This contradicts the convexity of $X$.
\end{proof}

The following is well-known \cite[16.3 Proposition, 16.4 Theorem]{Bredon}.

\begin{prop}\label{prop:homeo}
	Let $A$ be a bounded convex domain in $\mathbb{R}^n$ and let $o \in A$.
	Take $\epsilon > 0$ so that $B_{euc}(o,\epsilon) \subset A$.
	Define a map $\pi : \partial A \to \partial B_{euc}(o,\epsilon)$ 
	as a projection of $\partial A$ to $\partial B_{euc}(o,\epsilon)$ toward $o$.
	Then $\pi$ is a homeomorphism and can be extended to 
	a homeomorphism from $\overline{A}$ to $B_{euc}(o,\epsilon)$.
	In particular $\overline{A}$ is contractible.
\end{prop}

%	2-2
%------------------------------------------
\subsection{Hilbert geometry}\label{sec:hilb}
Let $X \subset \mathbb{R}^n$ ($n \ge 2$) be a non-empty bounded convex domain.
For any different two points $x,y \in X$ a line passing through $x$ and $y$ crosses $\partial X$ 
at just two points $x',y'$ where $x',x,y,y'$ are arranged in this order.
Such a chord $[x',y']$ is uniquely determined for $x$ and $y$.
The value 
\[
	\frac{|x y'||y x'|}{|x x'||y y'|}
\]
is called the \emph{cross ratio} of $x$ and $y$.
The cross ratio induces a metric $d$ on $X$ by 
\[
	d(x,y) = 
	\begin{cases}
	\ \log \frac{|x y'||y x'|}{|x x'||y y'|} & \text{\ if\ } x \neq y,\\
	\ 0 & \text{\ if\ } x = y,
	\end{cases}
\]
(for example, \cite{Busemann, de la Harpe}).
We call $d$ the {\em{Hilbert metric}} and $(X,d)$ a {\em{Hilbert geometry}}.

We say that a finite set of points in $\mathbb{R}^n$ is \emph{collinear} if it belongs to a single line.
We always suppose that elements of a collinear set $\{x_1,\ldots,x_k\}$ are arranged by their indices.
The following invariance of the cross ratio under 
the perspective projection is well-known \cite[Proposition 5.6.4]{Papadopoulos}.
See Figure \ref{fig:3}.

\begin{prop}\label{prop:2}		
	Let $\{a_1,a_2,a_3,a_4\}$ and $\{b_1,b_2,b_3,b_4\} \subset \mathbb{R}^n$ be collinear and 
	consist of distinct points, respectively. 
	If we have four lines $R_i$ passing through $a_i$ and $b_i$ $(i=1,\ldots,4)$
	which meet at a point $p \in \mathbb{R}^n$ or are parallel, then 
	\[
		\frac{|a_2a_4||a_3a_1|}{|a_2a_1||a_3a_4|} = \frac{|b_2b_4||b_3b_1|}{|b_2b_1||b_3b_4|}.
	\]
\end{prop}

\vspace{-10pt}
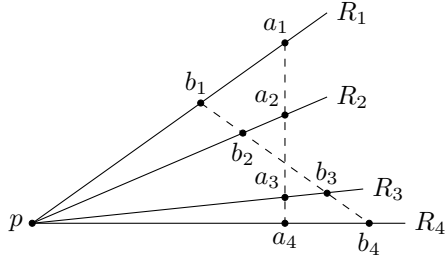
\begin{figure}[htb!]
\centering
\begin{tikzpicture}[scale=.8]
	\coordinate (P) at (0,0);
	
	\coordinate (A1) at (4.2,3);
	\coordinate (A2) at (4.2,1.8);
	\coordinate (A3) at (4.2,.43);
	\coordinate (A4) at (4.2,0);
	
	\coordinate (B1) at (2.8,2);
	\coordinate (B2) at (3.5,1.5);
	\coordinate (B3) at (4.9,0.5);
	\coordinate (B4) at (5.6,0);
	
	\coordinate (R1) at (4.9,3.5);
	\coordinate (R2) at (4.9,2.1);
	\coordinate (R3) at (5.5,.57);
	\coordinate (R4) at (6.2,0);
	
	\filldraw (P) circle[radius=0.5mm] node [anchor=east]{$p$};
	\filldraw (A1) circle[radius=0.5mm] node [anchor=south]{$a_1\ \ $};
	\filldraw (A2) circle[radius=0.5mm] node [anchor=south]{$a_2\ \ \ \ $};
	\filldraw (A3) circle[radius=0.5mm] node [anchor=south]{$a_3\ \ \ \ $};
	\filldraw (A4) circle[radius=0.5mm] node [anchor=north]{$a_4$};
	\filldraw (B1) circle[radius=0.5mm] node [anchor=south]{$b_1\ $};
	\filldraw (B2) circle[radius=0.5mm] node [anchor=north]{$b_2$};
	\filldraw (B3) circle[radius=0.5mm] node [anchor=south]{$b_3$};
	\filldraw (B4) circle[radius=0.5mm] node [anchor=north]{$b_4$};
	\filldraw (R1) node [anchor=west]{$R_1$};
	\filldraw (R2) node [anchor=west]{$R_2$};
	\filldraw (R3) node [anchor=west]{$R_3$};
	\filldraw (R4) node [anchor=west]{$R_4$};
	
	\draw (P) -- (R1);
	\draw (P) -- (R2);
	\draw (P) -- (R3);
	\draw (P) -- (R4);
	\draw [dashed](A1) -- (A4);
	\draw [dashed](B1) -- (B4);
\end{tikzpicture}
\caption{Proposition \ref{prop:2}.}
\label{fig:3}
\end{figure}

We recall some basic facts about Hilbert geometry.

\begin{thm}[\cite{Busemann,de la Harpe}]\label{thm:2}
	Let $X \subset \mathbb{R}^n$ be a non-empty bounded convex domain.
	\begin{itemize}
	\item[$\mathrm{(i)}$] 
		The Hilbert metric and the restricted Euclidean metric give the same topology on $X$.
	\item[$\mathrm{(ii)}$]
		The Hilbert geometry $(X,d)$ is a proper metric space.
	\item[$\mathrm{(iii)}$] 
		Every segment in $X$ is a geodesic in $(X,d)$.
	\item[$\mathrm{(iv)}$]
		The Hilbert geometry $(X,d)$ is uniquely geodesic if and only if there is 
		no pair of non-trivial segments $I,J$ in $\partial X$ 
		such that $I,J$ span an affine plane in $\mathbb{R}^n$.
		In particular if $\overline{X}$ is properly convex in $\mathbb{R}^n$, 
		then $(X,d)$ is uniquely geodesic.
	\end{itemize}
\end{thm}

%------------------------------------------
%	3:Uniform contractibility
%------------------------------------------
\section{Uniform contractibility}
We prove that every Hilbert geometry is uniformly contractible.

\begin{dfn}\label{uniformly contractible}
	A metric space $(Y,d_Y)$ is {\em{uniformly contractible}} if for any $R > 0$
	there exists $S \ge R$ such that any closed $R$-ball $B_Y(y,R)$ is contractible to a point in $B_Y(y,S)$.
\end{dfn}

\begin{lem}\label{lem:9}
	Let $(X,d)$ be a Hilbert geometry.
	Then for any $o,x,y \in X$ and any $z \in [x,y]$, we have 
	\[
		d(o,z) \le \max\left\{ d(x,o), d(y,o) \right\}.
	\]
	In particular, every open ball in $(X,d)$ is convex with respect to the Euclidean metric.
\end{lem}

\begin{proof}
	If $o,x,y$ are collinear then the claim is trivial and hence
	it suffices to consider the case where $o,x,y$ span a plane $H$.
	Let $[x_1,x_2]$ and $[y_1,y_2]$ be two chords through $o,x$ and $o,y$ in this order respectively.
	Then $[x_1,x_2]$ and $[y_1,y_2]$ belong to the plane $H$.
	For any $z \in [x,y]$ we take a chord $[z_1,z_2]$ passing through $o,z$ in this order.
	Since $\overline{X}$ is also convex, 
	$[z_1,z_2]$ intersects $[x_1,y_1]$ and $[x_2,y_2]$ at two points $z'_1, z'_2 \in H \cap \overline{X}$ respectively.
	If $[x_1,y_1]$ and $[x_2,y_2]$ are not parallel, 
	let $p \in H$ be the point on which the line including $[x_1,y_1]$ intersects 
	the line including $[x_2,y_2]$.
	In the case where $[x_1,y_1]$ and $[x_2,y_2]$ are parallel, we take the point at infinity as $p$.
	Considering the line $L$ through $p,z$, we see that 
	$L$ crosses $[x_1,x_2]$ and $[y_1,y_2]$ at $x_3$ and $y_3$ respectively.
	Since $z$ belongs to the segment $[x,y]$, 
	(i) $|ox| \ge |ox_3|$ and $|o y_3|\ge|o y|$ or 
	(ii) $|o y| \ge |o y_3|$ and $|ox_3|\ge|ox|$ must be satisfied.
	If (i) happens, by Proposition \ref{prop:2} we have 
	\[
		\frac{|ox_2||xx_1|}{|ox_1||xx_2|} \ge \frac{|ox_2||x_3x_1|}{|ox_1||x_3x_2|} 
									     =   \frac{|oz'_2||z z'_1|}{|oz'_1||z z'_2|}
									     \ge \frac{|oz_2||z z_1|}{|oz_1||z z_2|}.
	\]
	Thus we conclude that $d(o,x) \ge d(o,z)$.
	A similar inequality gives $d(o,y) \ge d(o,z)$ in the case where (ii) happens.
\end{proof}

\begin{rem}\label{rem:2}
	For $A > 0, B \ge A$ and $C \ge 0$, we have $B/A \ge (B+C)/(A+C)$.
\end{rem} 

\begin{prop}\label{prop:3}
	Every Hilbert geometry $(X,d)$ is uniformly contractible.
\end{prop}

\begin{proof}
	Every open ball of $(X,d)$ is convex with respect to the Euclidean metric by Lemma \ref{lem:9}.
	Thus by Theorem \ref{thm:2} (i) and Proposition \ref{prop:homeo}, any closed ball with respect to $d$ 
	is contractible in itself.
\end{proof}

%------------------------------------------
%	4:coarse bounded geometry
%------------------------------------------
\section{Coarse bounded geometry}
We show that Hilbert geometries have coarse bounded geometry.

\begin{dfn}\label{coarse bounded geometry}
	A metric space $(Y,d_Y)$ is said to be with {\em{coarse bounded geometry}} if there exists 
	$\epsilon > 0$ satisfying the following: For any $R > 0$,
	\[
		\sup\left\{\ l\ \vert\ y \in Y,\ y_1,\dots,y_l \in B_Y(y,R),\ i \neq j,\ d_Y(y_i,y_j) > \epsilon\ \right\} 
		< \infty.
	\]
\end{dfn}

In order to count points in a closed ball we compute the ratio of the volume of closed balls.

\begin{lem}\label{lem:5}
	Let $(X,d)$ be a Hilbert geometry.
	For any $R, r > 0$ $(R \ge r)$,  
	let $B_R$ and $B(x,r)$ be closed balls with $B(x,r) \subset B_R$.
	Then there exists a constant $1> D > 0$ such that 
	the map $f_x: \mathbb{R}^n \to \mathbb{R}^n$ given by 
	\[
		f_x(y) = x+D(y-x)
	\]
	sends $B_R$ into $B(x,r)$.
\end{lem}

\begin{proof}
	Fixing $y \in B_R \setminus B(x,r)$ arbitrarily,
	we take a chord $[x', y']$ passing through $x,y$ in this order.
	Let $z \in \partial B(x,r)$ be the intersection point of the segment $[x,y]$ and $\partial B(x,r)$.
	See Figure \ref{fig:1}.
	
\begin{figure}[htb!]
\begin{center}
\begin{tikzpicture}[scale=.8]
	\coordinate (A) at (0,0);
	\coordinate (A') at (.2,-.05);
	\coordinate (A") at (.5,-.07);
	\coordinate (X) at (1.05,-.16);
	\coordinate (Z) at (2.1,.73);
	\coordinate (ETA) at (-.5,1.97);
	\coordinate (XI) at (2.,-1.5);

	\coordinate (SP) at (-2.9,1.4);
	\coordinate (BR) at (-2.2,.45);
	\coordinate (Br) at (-1.5,0.);

	\filldraw (X) circle[radius=0.5mm] node [anchor=south]{$\ \ x$};	
	\filldraw (ETA) circle[radius=0.5mm] node [anchor=south]{$y'\ \,$};
	\filldraw (XI) circle[radius=0.5mm] node [anchor=north]{$x'$};

	\draw [line width=.6pt, name path=1] (A') ellipse (2.55cm and 1.6cm);
	\draw [line width=.5pt, name path=2] (A") ellipse (2cm and 1cm);
	\draw [line width=.6pt, name path=3, color=white] (A') ellipse (2.5cm and 1.3cm);
	\draw [line width=.5pt, name path=0] (XI) -- (ETA);
	\draw [line width=.7pt] (A) ellipse (3cm and 2cm);

	\draw (SP) node [anchor=west]{$X$};
	\draw (BR) node [anchor=west]{\small{$B_R$}};
	\draw (Br) node [anchor=west]{\small{$B(x,r)$}};
	
	\filldraw [name intersections={of=0 and 3, by=Y}] (Y) circle[radius=0.5mm] node [anchor=east]{$y$};
	\filldraw [name intersections={of=0 and 2, by=Z}] (Z) circle[radius=0.5mm] node [anchor=north]{$z\ \ $};
\end{tikzpicture}
\end{center}
\caption{Proof of Lemma \ref{lem:5}.}
\label{fig:1}
\end{figure}
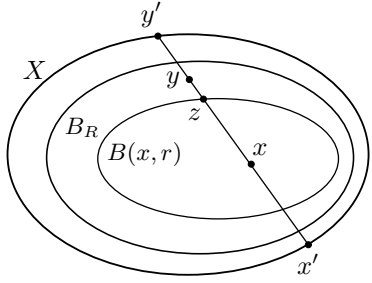

	Since $x, y \in B_R$, their distance is at most $2R$:
	\begin{equation}\label{eq:2-1}
		\frac{|y' x||x' y|}{|y' y||x' x|}= e^{d(x,y)} \le e^{2R}.	
	\end{equation}
	By expanding \eqref{eq:2-1} with $|y' x| = |y' y| + |x y|$, $|x' y| = |x' x| + |x y|$
	and $|y' x'|=|y' y|+|y x|+|x x'|$, we obtain
	\[
		|x y||y' x'| \le \left(e^{2R}-1\right)|y' y||x' x|.
	\]
	In a similar way,
	\[
		|x z||y' x'| = \left(e^{r}-1\right)|y'z||x' x|.
	\]
	Since $|y' y| \le |y' z|$ by our assumption,
	we get
	\[
		\frac{|x z|}{|x y|}
		\ge \frac{\left(e^{r}-1\right)|y'z||x' x|}{\left(e^{2R}-1\right)|y' y||x' x|}		 
		\ge \frac{e^{r}-1}{e^{2R}-1}.
	\]
	Defining $D = (e^{r}-1)/(e^{2R}-1)$, we have the conclusion.
\end{proof}

\begin{prop}\label{prop:1}
	Every Hilbert geometry is with coarse bounded geometry.
\end{prop}

\begin{proof}
	Fix $\epsilon > 0$ arbitrarily, let $R > 0$ and take any closed ball $B(x,R)$.
	Choose $\{x_1,\ldots,x_l\} \subset B(x,R)$ so that 
	$d(x_i,x_j) > 2\epsilon$ ($i \neq j$) for some $l \in \mathbb{N}$.
	This condition is equivalent to $B(x_i,\epsilon) \cap B(x_j,\epsilon) = \emptyset$ for each $i \neq j$.
	
	Since $B(x_i,\epsilon) \subset B(x,R+\epsilon)$, 
	there exists a contracting constant $D$ and a map $f_i : x \mapsto x_i + D(x - x_i)$ 
	such that $f_i(B(x,R+\epsilon)) \subset B(x_i,\epsilon)$ for each $i \in \{1,\ldots,l\ \}$ by Lemma \ref{lem:5}.
	Note that $D$ depends only on $R$ and $\epsilon$.
	We denote by $\mu$ the Lebesgue measure on $\mathbb{R}^n$.
	Then we have $\mu(f_i(B(x,R+\epsilon))) = D^n\mu(B(x,R+\epsilon))$ for all $i$.
	Since $\bigsqcup_{i=1}^l B(x_i,\epsilon) \subset B(x,R+\epsilon)$, 
	\begin{equation}\label{eq:2-2}
		\sum_{i = 1}^l \mu(B(x_i,\epsilon))
			 = \mu\left(\bigsqcup_{i=1}^l B(x_i,\epsilon)\right) \le \mu(B(x,R+\epsilon)).
	\end{equation}
	On the other hand, we also have
	\begin{equation}\label{eq:2-3}
			\sum_{i = 1}^l \mu(B(x_i,\epsilon)) 
			\ge \sum_{i=1}^l \mu(f_i(B(x,R+\epsilon))) 
			= l D^n\mu(B(x,R+\epsilon)).
	\end{equation}
	Combining \eqref{eq:2-2} and \eqref{eq:2-3}, we obtain
	\[
		l \le \frac{1}{D^n}.
	\]
	Thus we conclude
	\[
		\sup\left\{\ 
		l\ \left\vert\ x \in X,\ x_1,\ldots,x_l \in B(x,R),\ i\neq j,\ d(x_i,x_j) > 2\epsilon\ 
		\right.\right\}
		\ \le\ \frac{1}{D^n}.
	\]
	This completes the proof.
\end{proof}

%------------------------------------------
%	5:Corona
%------------------------------------------
\section{Corona}
The natural boundary of a bounded convex domain $X$ gives a compactification of the Hilbert geometry $(X,d)$.
This is a consequence of Theorem \ref{thm:2} (i).

\begin{cor}\label{cor:closure}
	Let $(X,d)$ be a Hilbert geometry.
	Then the closure $\overline{X}$ of $X$ in $\mathbb{R}^n$ gives a compactification of $(X,d)$.
\end{cor}	

We discuss when the natural boundary of a Hilbert geometry is a corona.
Let $(Y,d_Y)$ be a proper metric space.
A bounded continuous function $f:Y \to \mathbb{C}$ is a \emph{Higson function} on $Y$
if for any $\epsilon>0$ and any $C > 0$, there exists a bounded set $B \subset Y$
such that for $x,y\in Y$ with $d_Y(x,y) \le C$ and $x \not\in B$ we have $\left|f(x)-f(y)\right|<\epsilon$.

\begin{dfn}\label{corona}
	A metrizable compactification $\overline{Y}$ of a proper metric space $Y$ is a \emph{coarse compactification}
	if the restriction of every continuous function on $\overline{Y}$ is a Higson function on $Y$.
	We call the boundary $\overline{Y}\setminus Y$ a \emph{corona} of $Y$.
\end{dfn}

The following is self-evident by the definition of the Hilbert metric.

\begin{lem}\label{lem:bounded}
	Let $(X,d)$ be a Hilbert geometry.
	If a subset $A \subset X$ satisfies $d_{euc}(A,\partial X) > \epsilon$ for some $\epsilon > 0$ then it is bounded, 
	i.e., the diameter of $A$ with respect to $d$ is finite.
\end{lem}

\begin{lem}\label{lem:3}
	For two converging sequences $x_i \rightarrow x$ and $y_i \rightarrow y$ in $\mathbb{R}^n$,
	any sequence $\{z_i\}_i$ consisting of $z_i \in [x_i,y_i]$ $(i \in \mathbb{N})$ has a converging subsequence.
	Moreover the limit point lies in $[x,y]$.
\end{lem}

\begin{proof}
	From the assumption, there exists $t_i \in [0,1]$ such that 
	$z_i = t_i x_i + (1-t_i) y_i$ for each $i \in \mathbb{N}$.
	By taking a subsequence we may assume that $\{t_i\}_i$ converges to some point $t \in [0,1]$.
	Set $z = t x + (1-t)y$.
	Then we have
	\[
		\|z_i-z\| \le \|t_i(x_i-x)\|+\|(t_i-t)x\|+\|(1-t_i)(y_i-y)\|+\|(t-t_i)y\| \longrightarrow 0
	\]
	by the triangle inequality.
\end{proof}

The next lemma is the key to proving Theorem \ref{thm:main-1}.

\begin{lem}\label{lem:1}
	Let $(X,d)$ be a Hilbert geometry with a base point $o \in X$.
	Suppose that $\overline{X}$ is properly convex in $\mathbb{R}^n$. 
	Then for any $C > 0$ and $\delta > 0$ there exists a bounded subset $M_{\delta,C}$ satisfying the following:
	For any $x,y \in X$ if $x \not\in M_{\delta,C}$ and $d(x,y) \le C$ then $|x y| < \delta$.
\end{lem}

\begin{proof}
	Take $x,y \in X$ arbitrarily and suppose that a chord $[\xi, \eta]$ passes through $x,y$ in this order.
	By the definition of the Hilbert metric we have
	\[
		e^{d(x,y)} = \frac{(|x y| + |x\xi|)(|x y|+|y\eta|)}{|x\xi||y\eta|}.
	\]
	If $d(x,y) \le C$ and $\delta \le |x y|$, then we obtain 
	\[
		e^C \ge \frac{(|x y| + |x\xi|)(|x y| + |y\eta|)}{|x\xi||y\eta|}
			\ge \frac{\delta^2}{|x\xi||y\eta|}.	
	\]
	Let $\mathrm{diam_{euc}(\overline{X})}$ be the diameter of $\overline{X}$ with respect to the Euclidean metric.
	By putting $E = \delta^2 / (e^C\cdot\mathrm{diam_{euc}}(\overline{X}))$ we see that 
	$|x\xi|,|y\eta|$ should satisfy $|x\xi|, |y\eta| \ge E$ 
	since $|x \xi|,|y \eta| \le |\xi \eta| \le \mathrm{diam_{euc}}(\overline{X})$.
	We also have $|x \eta| = |x y| + |y \eta| \ge E$. 
	
	With this in mind, we consider a set 
	\[
		M_{\delta,C} = \left\{\ x \in X\ \left\vert\ \exists\, y \in X, \ |x y| \ge \delta
		\text{\ and\ } d(x,y) \le C\ \right.\right\}.
	\]
	We claim that $M_{\delta,C}$ is bounded.
	Then the assertion follows.
	To verify this claim, we assume that $M_{\delta,C}$ is not bounded. 
	Then by Lemma \ref{lem:bounded} for a fixed decreasing sequence $\epsilon_i \to 0$ 
	there exists a sequence $\{x_i\}_i$ in $M_{\delta,C}$ satisfying 
	$\inf_{\zeta \in \partial X} |x_i\zeta| \le \epsilon_i$ for each $i$.
	By the definition of $M_{\delta,C}$, we have a sequence $\{y_i\}_i$ in $X$ such that 
	$|x_i y_i| \ge \delta$ and $d(x_i,y_i) \le C$.
	Applying the above argument to $x_i$ and $y_i$, 
	we get $\xi_i, \eta_i \in \partial X$ so that $x_i \in [\xi_i,\eta_i]$
	and $|x_i\xi_i|, |x_i\eta_i| \ge E$ for each $i$.
	By taking subsequences we may assume $\xi_i,\eta_i \to \xi_\infty, \eta_\infty$ ($i \to \infty$) for some 
	$\xi_\infty, \eta_\infty \in \partial X$.
	Lemma \ref{lem:3} shows that the corresponding subsequence $\{x_i\}_i$ has an accumulation point 
	$x_\infty \in [\xi_\infty,\eta_\infty]$.
	Furthermore, $x_\infty \neq \xi_\infty, \eta_\infty$ because 
	$|x_i\xi_i|,  |x_i\eta_i| \ge E$ for all $i \in \mathbb{N}$.
	However since $\inf_{\zeta \in \partial X} |x_\infty\zeta| \le \epsilon_i \to 0$, 
	we have $x_\infty \in \partial X$.
	This contradicts Lemma \ref{lem:4}.
\end{proof}

\begin{proof}[Proof of Theorem \ref{thm:main-1}]
	Fix $o \in X$ to be a base point.
	Suppose that $\partial X$ includes a non-trivial segment $[\alpha,\beta]$. 
	We take two distinct points $\xi,\eta$ in $[\alpha,\beta] \setminus \{\alpha,\beta\}$ 
	and a continuous function $f$ on $\overline{X}$ 
	which separates $\xi$ and $\eta$, i.e., $f(\xi) \neq f(\eta)$. 
	We show that the Hausdorff distance with respect to $d$ between
	$[o,\xi] \setminus \{\xi\}$ and $[o,\eta] \setminus \{\eta\}$ is finite.
	Then it follows that $f$ is not a Higson function and hence $\partial X$ is not a corona.
	Suppose that $|\alpha \xi| < |\alpha \eta|$.
	For any $x \in [o,\xi] \setminus \{\xi\}$ let $y \in [o,\eta] \setminus \{\eta\}$ be the point 
	such that the line $L$ through $x$ and $y$ is parallel to the line through $\alpha$ and $\beta$.
	Let $x', y' \in X$ be the intersection points of $L$ and $[o,\alpha]$, $[o,\beta]$ respectively.
	Then by Proposition \ref{prop:2} and Remark \ref{rem:2} we have
	\[
		\frac{|\xi \beta||\eta \alpha|}{|\xi \alpha||\eta \beta|}
		= \frac{|x y'||y x'|}{|x x'||y y'|}
		\ge e^{d(x,y)}.
	\]
	By symmetry, we have the conclusion.
	
	To show the converse,
	we note that $\overline{X}$ is a compact metric space with respect to the restriction of the Euclidean metric. 
	Since any continuous function on $\overline{X}$ is uniformly continuous,
	the assertion follows from Lemma \ref{lem:1}.
\end{proof}

Since the following argument is standard in coarse geometry, we write briefly
(for the precise definition, see \cite{H-R95, F-O13}). 

\begin{proof}[Proof of Corollary \ref{cor:1}]
	Since $\overline{X}$ is a coarse compactification, the transgression map $T_{\partial X}$ and 
	the Higson-Roe map $b_{\partial X}$ are well-defined and the following diagram is commutative:
	\[
	\begin{tikzcd}
	K_\bullet(X) \ar{dr}{c(X)} \ar{rr}{A(X)} \ar{ddr}[swap]{\partial_{\partial X}} 
	& & K_\bullet(C^\ast X) \ar{ddl}{b_{\partial X}} \\
	& KX_\bullet(X) \ar{ur}{\mu(X)} \ar[xshift=-1ex]{d}{T_{\partial X}} & \\
	& \widetilde{K}_{\bullet-1}(\partial X) & \hspace{3em}.
	\end{tikzcd}
	\]
	Here $K_\bullet(X)$ is the $K$-homology of $X$, 
	$C^\ast X$ is the Roe algebra of $X$, 
	$K_\bullet(C^\ast X)$ is the $K$-theory of $C^\ast X$, 
	$A(X)$ is the assembly map for $X$, 
	$c(X)$ is the coarsening map for $X$, 
	$\mu(X)$ is the coarse assembly map for $X$, 
	$\widetilde{K}_{\bullet-1}(\partial X)$ is the reduced $K$-homology of $\partial X$ and 
	$\partial_{\partial X}$ is a connecting map of the $K$-homology
	(see \cite[Section 6]{H-R95} and also \cite[Section 1]{F-O13}).
	 
	By Proposition \ref{prop:homeo}, $\partial_{\partial X}$ is an isomorphism.	
	Combining Propositions \ref{prop:3} and \ref{prop:1}, we see that $c(X)$ is also an isomorphism
	(see \cite[Proof of Theorem 4.8]{Emerson-Meyer} and also \cite[Section 3.2]{F-O13}).
	By tracing the diagram, we have the conclusion.
\end{proof}

\begin{rem}
	On the above setting, 
	the transgression map $T_{\partial X}$ is an isomorphism and the assembly map $A(X)$ is injective.
	We note that the coarsening map $c(X)$ is an isomorphism for any Hilbert geometry because 
	Proposition \ref{prop:1} and \ref{prop:3} do not require the domain is properly convex.
\end{rem}

%------------------------------------------
%	6:Asymptotic dimension
%------------------------------------------
\section{Asymptotic dimension}
We begin with the definition of the asymptotic dimension of a metric space.
For a family $\mathcal{U}$ of subsets of a metric space, 
the \emph{$r$-multiplicity} of $\mathcal{U}$ is the smallest number $n$ such that 
the every closed $r$-ball intersects at most $n$ elements of $\mathcal{U}$.
There are several equivalent ways to define the asymptotic dimension 
(see, for example, \cite[\S 3]{Bell-Dranishnikov}).
In this paper, we adopt the following definition.

\begin{dfn}\label{asymptotic dimension}
	Let $Y$ be a metric space.
	We say that the \emph{asymptotic dimension} $\mathrm{asdim}(Y)$ of $Y$ does not exceed $m$ if 
	for each $r > 0$ there exists a uniformly bounded cover $\mathcal{U}$ with $r$-multiplicity $\le m+1$.
	In this case we write $\mathrm{asdim}(Y) \le m$.
	If $\mathrm{asdim}(Y) \le m$ but $\mathrm{asdim}(Y) \not\le m-1$, 
	then we say that the asymptotic dimension of $Y$ is $m$.
\end{dfn}

%	6-1
%------------------------------------------
\subsection{Lower bound}

To get the lower bound we use the coarse cohomology.
For a proper metric space $Y$, $HX^m(Y)$ and $H_c^m(Y)$ denote the $m$-dimensional coarse cohomology and 
Alexander-Spanier (or equivalently, \v{C}ech) cohomology with compact supports of $Y$ respectively.
It is known that there naturally exists a \emph{character map} $c^m(Y) : HX^m(Y) \to H_c^m(Y)$.
See \cite{Roe1} and \cite{Roe2} for details.
Since the argument is standard in coarse geometry, we write briefly. 

\begin{lem}\label{lem:HX}
	Let $Y$ be a proper metric space.
	If $HX^m(Y)$ is not trivial and
	the character map $c^m(Y)$ is injective,
	then we have $m \le \mathrm{asdim}(Y)$.
\end{lem}

\begin{proof}
	We can assume that $m' = \mathrm{asdim}(Y) < \infty$.
	Then we have an anti-\v{C}ech system $\{\mathcal{U}_k\}_k$ such that
	each nerve complex $|\mathcal U_k|$ satisfies $H_c^m(|\mathcal{U}_k|)=0$
	for any $m > m'$ by \cite[Theorem 9.9(c)]{Roe2}.
	We take a partition of unity $\rho$ subordinate to the cover $\mathcal{U}_k$ of $Y$ for some $k$.
	Then it defines a proper continuous map $\kappa : Y \to |\mathcal{U}_k|$.
	Since $|\mathcal{U}_k|$ admits a proper metric
	which is coarsely equivalent to $Y$ by $\kappa$ \cite{Wright}
	and the coarse cohomology is coarsely invariant,
	the character map $c^m(Y) : HX^m(Y) \to H_c^m(Y)$ factors through $H_c^m(|\mathcal{U}_k|)$.
	The map must be a $0$-map if $m > m'$.
	If the character map $c^m(Y)$ is injective,
	then we have $HX^m(Y)=0$.
\end{proof}

\begin{prop}\label{prop:below}
	The asymptotic dimension of any $m$-dimensional Hilbert geometry is at least $m$.
\end{prop}

\begin{proof}
	Let $(X,d)$ be an $m$-dimensional Hilbert geometry. 
	Since $X$ is uniformly contractible (Proposition \ref{prop:3}) 
	we see that the character map is an isomorphism \cite[(3.33) Proposition]{Roe1}.
	On the other hand we have $H_c^m(X) = \mathbb{R}$ because $X$ is homeomorphic to $\mathbb{R}^m$.
	By Lemma \ref{lem:HX} we have $\mathrm{asdim}(X) \ge m$.
\end{proof}

%	6-2
%------------------------------------------
\subsection{Lemmas}

Let $X \subset \mathbb{R}^n$ be a non-empty bounded convex domain and $(X,d)$ the Hilbert geometry.
Fix $o \in X$ to be a base point.
We define a \emph{ray} $\ell$ as (the image of) an isometric embedding from $[0,\infty)$ into $X$
such that its image is included in a line of $\mathbb{R}^n$ and $\ell(0) = o$.

\begin{lem}\label{lem:10}
	Let $(X,d)$ be a Hilbert geometry with a base point $o$.
	For $a_2,b_2 \in X$, 
	take two chords $[a_1,a_3]$ and $[b_1,b_3]$ passing through $o, a_2$ and $o,b_2$ in this order respectively.
	Let $L_i$ $(i= 1,2,3)$ be the line through $a_i$ and $b_i$.
	If $o,a_2,b_2$ are not collinear and satisfy $d(o,a_2) = d(o,b_2)$,
	then three lines $L_1,L_2,L_3$ meet at one point in $\mathbb{R}^n \setminus X$ or are parallel.
\end{lem}

\begin{proof}
	Note that seven points $a_1, a_2, a_3, b_1, b_2, b_3 ,o$ are on the same plane in $\mathbb{R}^n$.
	We assume that $L_1$ and $L_3$ are not parallel.
	Let $p$ be the point where $L_1$ and $L_3$ intersect.
	By the choice of $a_1, a_3, b_1, b_3$, two chords $[a_1,b_1]$ 
	and $[a_3,b_3]$ do not intersect each other in $X$.
	Since $X$ is convex, $p$ is not contained in $X$. 
	Consider a line $L'$ through $b_2$ and $p$. 
	Then $L'$ crosses the segment $[o,a_3]$ at a point $q$ (Figure \ref{fig:4}).
	By Proposition \ref{prop:2} we see that $d(o,b_2) = d(o,q)$ and thus $d(o,a_2) = d(o,q)$.
	This shows that $a_2 = q$ because $o,a_2,q$ are on the same segment $[o,a_3]$.
	For the case where $L_1$ and $L_3$ are parallel, we can show the lemma in a similar way.
\end{proof}

\vspace{-10pt}

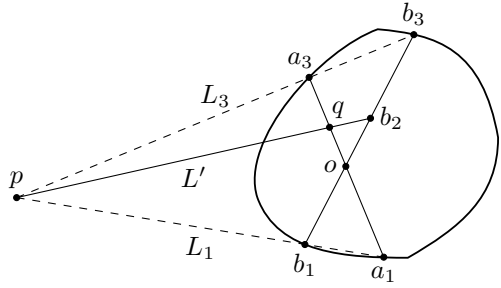
\begin{figure}[htb!]
\begin{center}
\begin{tikzpicture}[scale=.8]
	\coordinate (A) at (0,0);
	\coordinate (B) at (2.5,-1);
	\coordinate (C) at (4,1);
	\coordinate (D) at (2,2.8);
	\coordinate (O) at (1.47,.52);
	\coordinate (A1) at (2.1,-.99);
	\coordinate (A2) at (1.2,1.17);
	\coordinate (A3) at (.86,2);
	\coordinate (B1) at (.79,-.78);
	\coordinate (B2) at (1.88,1.32);
	\coordinate (B3) at (2.6,2.71);
	\coordinate (P) at (-4,0);

	\filldraw (O) circle[radius=0.5mm] node [anchor=east]{$o$};
	\filldraw (A1) circle[radius=0.5mm] node [anchor=north]{$a_1$};
	\filldraw (A2) circle[radius=0.5mm] node [anchor=south]{$\ \ q$};	
	\filldraw (A3) circle[radius=0.5mm] node [anchor=south]{$a_3\ \ $};
	\filldraw (B1) circle[radius=0.5mm] node [anchor=north]{$b_1$};
	\filldraw (B2) circle[radius=0.5mm] node [anchor=west]{$b_2$};	
	\filldraw (B3) circle[radius=0.5mm] node [anchor=south]{$b_3$};
	\filldraw (P) circle[radius=0.5mm] node [anchor=south]{$p$};
 	
	\draw [line width=.7pt] (A) 
	.. controls +(290:1cm) and +(180:1cm) .. (B) 
	.. controls +(30:1cm) and +(270:1cm) .. (C)
	.. controls +(100:1.5cm) and +(355:1cm) .. (D)
	.. controls +(200:1cm) and +(108:1cm) .. (A);
	
	\draw (A1) -- (A3);
	\draw (B1) -- (B3);
	\draw[dashed] (P) to node[below] {$L_1$}  (A1);
	\draw[dashed] (P) to node[above] {$L_3$} (B3);
	\draw (P) to node[below] {$L'$}  (B2);
\end{tikzpicture}

\end{center}
\caption{Proof of Lemma \ref{lem:10}.}
\label{fig:4}
\end{figure} 

\begin{lem}\label{lem:6}
	For a pointed Hilbert geometry $(X,d,o)$, given two distinct rays $\ell_1$ and $\ell_2$,
	if $0 < s < t$ then 
	\[
		d(\ell_1(s),\ell_2(s)) \le d(\ell_1(t), \ell_2(t)).
	\]
\end{lem}

\begin{proof}
	Suppose that two chords $[a_t, b_t]$ and $[a_s, b_s]$ pass through 
	$\ell_1(t), \ell_2(t)$ and $\ell_1(s), \ell_2(s)$ in this order respectively.
	Applying Lemma \ref{lem:10} to points on $\ell_1$ and $\ell_2$ twice,
	we notice that $[a_t,b_t]$ and $[a_s,b_s]$ do not intersect each other in $X$.
	Putting this and the convexity of $X$ together we see that $[a_t,o]$ and $[b_t,o]$ cross $[a_s, b_s]$. 
	We denote individual intersection points by $a'_s, b'_s$ so that $a_s, a'_s, b'_s, b_s$ are arranged in this order.
	The segment $[a'_s, b'_s]$ contains $\ell_1(s)$ and $\ell_2(s)$.
	By Proposition \ref{prop:2}, we have
	\[
		e^{d(\ell_1(t), \ell_2(t))}
			= \frac{|\ell_1(t)b_t||\ell_2(t)a_t|}{|\ell_1(t)a_t||\ell_2(t)b_t|} 
			= \frac{|\ell_1(s)b'_s||\ell_2(s)a'_s|}{|\ell_1(s)a'_s||\ell_2(s)b'_s|}.
	\]
	On the other hand by Remark \ref{rem:2} we also have
	\[
		\frac{|\ell_1(s)b'_s||\ell_2(s)a'_s|}{|\ell_1(s)a'_s||\ell_2(s)b'_s|}
			\ge \frac{|\ell_1(s)b_s||\ell_2(s)a_s|}{|\ell_1(s)a_s||\ell_2(s)b_s|} 
			= e^{d(\ell_1(s),\ell_2(s))}.
	\]
	This shows the lemma.
\end{proof}

The next lemma follows from the triangle inequality and Lemma \ref{lem:6}.

\begin{lem}\label{lem:11}
	Let $(X,d,o)$ be a pointed Hilbert geometry.
	For $x,y \in X$ let $\ell_x$ and $\ell_y$ be rays passing through $x$ and $y$ respectively.
	If $d(x,y) \le r$ and $d(o,x) \le d(o,y)$ then $d(\ell_x(s), \ell_y(s)) \le 2r$ for $s \le d(o,y)$.
\end{lem}

\begin{proof}
	Note that for any $d(o,x) \le t \le d(o,y)$, we have 
	\[
		d(\ell_x(t),x) + d(y,\ell_y(t)) = d(o,y)-d(o,x) \le d(x,y) \le r,
	\]
	by the triangle inequality. Hence it holds
	\[
		d(\ell_x(t),\ell_y(t)) \le d(\ell_x(t),x) + d(x,y) + d(y,\ell_y(t)) \le 2r.
	\]
	By Lemma \ref{lem:6} we have $d(\ell_x(s), \ell_y(s)) \le 2r$ for any $s \le d(o,y)$.
\end{proof}

%	6-3
%------------------------------------------
\subsection{Upper bound}
Henceforth, we concentrate on 2-dimensional Hilbert geometries.
In such a case, the boundary of a ball with respect to the Hilbert metric is 
homeomorphic to the circle $\mathbb{S}^1$ by Lemmas \ref{prop:homeo} and \ref{lem:9}.
Because of this we may simply call the boundary of a ball a \emph{circle} and 
assume that each circle is endowed with the \emph{counterclockwise} (CCW) direction.
For distinct two points $a,b$ on a circle, the \emph{arc} $\widehat{a b}$ from $a$ to $b$ stands for the closed subpath on 
the circle from $a$ to $b$ in the CCW direction.

Let $R>0$. 
We consider two conditions for an arc $\widehat{ac}$:
\begin{itemize}
	\item[$\blacklozenge$]
	\emph{There exists a point $b \in \widehat{ac}$ such that $d(a,b) \ge R$.}
	\item[$\lozenge$]
	\emph{The diameter $\mathrm{diam}(\widehat{ac}) = \max_{x,y \in \widehat{ac}}d(x,y)$ is not larger than $4R$.}
\end{itemize}

\begin{lem}\label{lem:key}
	Let $R > 0$.
	If an arc $\widehat{ac}$ satisfies $\blacklozenge$ for $R$ then
	$\widehat{ac}$ can be decomposed into arcs 
	$\widehat{a_0a_1},\widehat{a_1a_2},\ldots,\widehat{a_{k-1}a_k}$ $(k \in 2\mathbb{Z}+1)$, 
	each of which satisfies $\blacklozenge$ and $\lozenge$ for $R$.
\end{lem}

\begin{proof}
	Note that $\widehat{ac}$ is homeomorphic to a closed interval.
	Since the distance function $d(a,-)\vert_{\widehat{ac}}$ is continuous, 
	there exists a unique point $a_1 \in \widehat{ac}$
	satisfying $d(a,a_1) = R$ and $d(a,b) < R$ for any $b \in \widehat{aa_1}\setminus \{a_1\}$.
	We have three cases:
	(i) $a_1 = c$.
	(ii) $a_1 \neq c$ and $\widehat{a_1c}$ does not satisfy $\blacklozenge$.
	(iii) $a_1 \neq c$ and $\widehat{a_1c}$ satisfies $\blacklozenge$.
	In the case (ii), we put $a_2=c$.
	For the case (iii), take $a_2 \in \widehat{a_1c}$ satisfying 
	$d(a_1,a_2) = R$ and $d(a_1,b) < R$ for any $b \in \widehat{a_1a_2}\setminus \{a_2\}$.
	
	By repeating this procedure, we have a sequence $a = a_0,a_1,\ldots, a_k = c$
	on $\widehat{ac}$ arranged in the CCW direction.
	For each $i$ ($i\neq k$), the arc $\widehat{a_{i-1}a_i}$ satisfies $\blacklozenge$ and $\lozenge$.
	If $\widehat{a_{k-1}a_k}$ does not satisfy $\blacklozenge$ then 
	$\widehat{a_{k-2}a_k}$ has $\blacklozenge$ and $\lozenge$.
	In such a case, we erase $a_{k-1}$ and rename $a_k$ to $a_{k-1}$.
	
	Finally if the number of resulting arcs is even then by erasing $a_1$ and renaming $a_i$ to $a_{i-1}$ for $i>1$ 
	we obtain the required arcs.
\end{proof}

\begin{prop}\label{prop:asdim}
	Let $(X,d)$ be a $2$-dimensional Hilbert geometry.
	Then the asymptotic dimension of $(X,d)$ is at most $2$.
\end{prop}

\begin{proof}
	We fix $R > 4r > 0$ arbitrarily and $o$ to be a base point of $X$.
	We set $A_0 = B(o,R)$.
	For each $i \in \mathbb{N}$, we define 
	\[
		A_i = \{\ x \in X \ \vert\ iR \le d(o,x) \le (i+1)R \ \}, \ \  
		S_i = \{\ x \in X \ \vert\ d(o,x) = iR \ \}.
	\] 
	For $i>1$, set a map $\pi_i : S_{i} \to S_{i-1}$ to be the projection toward $o$,
	which is a homeomorphism (see Lemma \ref{prop:homeo}).
	
	To construct a cover of $X$ we put markers on each $S_i$ inductively.
	For each step $i$, we would like to decompose $S_i$ into an even number of arcs
	\[
		\widehat{x^i_0y^i_0},\ \widehat{y^i_0x^i_1},\ \widehat{x^i_1y^i_1},\ \ldots,\ 
		\widehat{x^i_{k_i}y^i_{k_i}},\ \widehat{y^i_{k_i}x^i_0},
	\]
	each of which satisfies $\blacklozenge$ and $\lozenge$ for $R$.	
	For $i > 1$ we require the following:
	for any $j \in \mathbb{Z}/(k_i+1)\mathbb{Z}$ there exist $p,q \in \mathbb{Z}/(k_{i-1}+1)\mathbb{Z}$ 
	such that $\pi_i(x^i_j) = y^{i-1}_p$ and $\pi_i(y^i_j) = x^{i-1}_q$.
	We say that such a decomposition of $S_i$ is \emph{admissible}.
	
	\noindent\underline{\textbf{Step $1$\ :\ }}
	Since  $S_1$ is decomposed into two arcs with the half length of $S_1$, 
	we can construct an admissible decomposition of $S_1$ by Lemma \ref{lem:key}.
	
	\noindent\underline{\textbf{Step $i+1$\ :\ }}
	Suppose that we have an admissible decomposition of $S_i$.
	We decompose $S_{i+1}$ by the arcs $\pi_{i+1}^{-1}(\widehat{x^i_{j}y^i_{j}})$, 
	$\pi_{i+1}^{-1}(\widehat{y^i_{j}x^i_{j+1}})$ ($j \in \mathbb{Z}/(k_i+1)\mathbb{Z}$).
	Let $z^i_j := \pi_{i+1}^{-1}(x^i_j)$ and $w^i_j := \pi_{i+1}^{-1}(y^i_j)$.
	Since each arc satisfies $\blacklozenge$ from Lemma \ref{lem:6},
	it can be decomposed into an odd number of arcs satisfying $\blacklozenge$ and $\lozenge$ by Lemma \ref{lem:key}.
	Note that the number of resulting arcs is even.
	Label individual end points of the arcs as
	$x^{i+1}_0 := w^i_0$ and $y^{i+1}_0,x^{i+1}_1,\ldots, y^{i+1}_{k_{i+1}}$ in the CCW direction on $S_{i+1}$.
	Then we have an admissible decomposition of $S_{i+1}$.
	
	Put $U_{0,0} = A_0$ and define $U_{i,j}$ as a bounded closed set enclosed by 
	\[
		\widehat{x^i_j x^i_{j+1}}\cup [x^i_{j+1},z^i_{j+1}] \cup \widehat{z^i_j z^i_{j+1}} \cup [z^i_j,x^i_j].
	\] 
	See Figure \ref{fig:cover}. 
	Then $\mathrm{diam}(U_{i j}) \le 10R$ by $\lozenge$ and $\mathcal{U} = \{U_{i,j}\}_{i,j}$ is a cover of $X$.

%\vspace{-10pt}
\begin{figure}[htb!]
\begin{center}
\begin{tikzpicture}[scale=.5]
	\coordinate (O) at (0,0);
	\coordinate (B) at (0:5);
	\coordinate (C) at (60:5);
	\coordinate (D) at (60:7);
	\coordinate (E) at (0:7);
	\coordinate (F) at (30:7);
	\coordinate (H) at (40:7);
	\coordinate (I) at (40:5);
	\coordinate (J) at (30:8.5);
	\coordinate (K) at (40:8.5);
	\coordinate (L) at (15:7);
	\coordinate (M) at (15:8.5);
	
	\draw[fill=lightgray]
  		($(0,0) + (0:5)$) arc (0:60:5)
  		--
		($(0,0) + (60:7)$) arc (60:0:7)
  		-- cycle;
		
	\filldraw (O) circle[radius=0.5mm] node [anchor=east]{$o$};
	\filldraw (B) circle[radius=0.5mm] node [anchor=north]{$x^i_j$};
	\filldraw (C) circle[radius=0.5mm] node [anchor=east]{$x^i_{j+1}$};	
	\filldraw (D) circle[radius=0.5mm] node [anchor=south]{$\ \ z^i_{j+1}$};
	\filldraw (E) circle[radius=0.5mm] node [anchor=north]{$\ \ z^i_j$};
	\filldraw (F) circle[radius=0.5mm] node [anchor=north]{};
	\filldraw (H) circle[radius=0.5mm] node [anchor=north]{};
	\filldraw (L) circle[radius=0.5mm] node [anchor=north]{};
	\filldraw (I) circle[radius=0.5mm] node [anchor=north]{$y^i_j\ \ $};

	\draw[dashed] (O) -- (C);
	\draw[dashed] (O) -- (B);
	\draw[dashed] (F) -- (J);
	\draw[dashed] (I) -- (H);
	\draw (H) -- (K);
	\draw (L) -- (M);
\end{tikzpicture} 
\end{center} 
\caption{A piece $U_{i,j}$ of the cover $\mathcal{U}$.}
\label{fig:cover}
\end{figure}
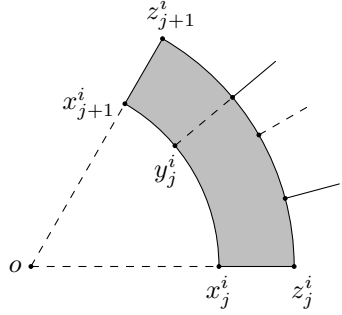

	We check that the $r$-multiplicity of $\mathcal{U}$ is at most $3$.
	Take a closed $r$-ball $B_r$ which is not included in the interior of $A_0$.
	Then there exists the smallest number $i$ so that $B_r$ is included in the interior of $A_{i-1} \cup A_i$
	since $2r < R$.
	Consider the set
	\[
		\Pi_i(B_r):= S_i \cap 
		\left\{\ \ell_x(t)\ \left\vert\  0 < t < \infty,\ \ell_x \text{\ is a ray through $x \in B_r$}\right.\ \right\}.
	\]
	Then we see that $\mathrm{diam}(\Pi_i(B_r)) \le 4r$ by Lemma \ref{lem:11}.
	Since $\Pi_i(B_r)$ is an arc on $S_i$, the inequality $4r < R$ implies that at most one of
	$x^i_0, y^i_0,\ldots,x^i_{k_i}, y^i_{k_i}$ is contained in $\Pi_i(B_r)$.
	Consequently, we have the following:
	(i) If $B_r \cap S_i = \emptyset$ then at most two elements of $\mathcal{U}$ intersect $B_r$.
	(ii) If $B_r \cap S_i \neq \emptyset$ then at most three elements of $\mathcal{U}$ intersect $B_r$.
\end{proof}

\begin{proof}[Proof of Theorem \ref{thm:main-2}]
	The assertion follows from Proposition \ref{prop:below} and \ref{prop:asdim}.
\end{proof}

%------------------------------------------
%------------------------------------------
%	References
%------------------------------------------
%------------------------------------------

\end{document}